\documentclass[12pt]{amsart}

\usepackage{color}

\newtheorem{theorem}{Theorem}[section]
\newtheorem{proposition}[theorem]{Proposition}
\newtheorem{lemma}[theorem]{Lemma}
\newtheorem{corollary}[theorem]{Corollary}

\theoremstyle{definition}

\newtheorem{conjecture}[theorem]{Conjecture}
\newtheorem{remark}[theorem]{Remark}

\newcommand{\ZZ}{ \ensuremath{\mathbb{Z}}}

\newcommand{\QQ}{ \ensuremath{\mathbb{Q}}}
\newcommand{\KK}{ \ensuremath{\mathbb{K}}}

\newcommand{\Hom}{\ensuremath{\mathrm{Hom}}\hspace{1pt}}

\newcommand{\aaa}{\mathbf{a}}
\newcommand{\bb}{\mathbf{b}}
\newcommand{\ttt}{\mathbf{t}}
\newcommand{\ee}{\mathbf{e}}

\newcommand{\Hilb}{\mathrm{Hilb}}
\newcommand{\lk}{{\mathrm{lk}}}
\newcommand{\st}{\mathrm{st}}

\def\cocoa{{\hbox{\rm C\kern-.13em o\kern-.07em C\kern-.13em o\kern-.15em A}}}

\begin{document}

\title[Balanced generalized lower bound inequality]{Balanced generalized lower bound inequality\\ for simplicial polytopes}

\author{Martina Juhnke-Kubitzke}
\address{
Martina Juhnke-Kubitzke, 
FB 12-- Institut f\"ur Mathematik,
Goethe-Universit\"at Frankfurt,
Robert-Mayer Str. 10, 60325 Frankfurt am Main, Germany
}
\email{kubitzke@math.uni-frankfurt.de}

\author{Satoshi Murai}
\address{
Satoshi Murai,
Department of Pure and Applied Mathematics,
Graduate School of Information Science and Technology,
Osaka University,
Toyonaka, Osaka, 560-0043, Japan
}
\email{s-murai@ist.osaka-u.ac.jp
}



\begin{abstract}
A remarkable and important property of face numbers of simplicial polytopes 
is the generalized lower bound inequality, which says that the $h$-numbers 
of any simplicial polytope are unimodal. Recently, for balanced simplicial $d$-polytopes, that is  simplicial $d$-polytopes whose underlying graphs are $d$-colorable, Klee and Novik proposed a balanced analogue of this inequality, that is stronger than just unimodality. 
The aim of this article is to prove this conjecture of Klee and Novik. 
For this, we also show a Lefschetz property for rank-selected subcomplexes of balanced
simplicial polytopes and thereby obtain new inequalities for their $h$-numbers.
\end{abstract}

\maketitle

\section{Introduction}
The study of face numbers of convex polytopes
is one of the main themes in algebraic and geometric combinatorics and has attracted a lot of attention during the last decades. 
It has been of great interest to completely characterize the possible face numbers of simplicial polytopes and to find sufficient and necessary conditions for face numbers of classes of simplicial complexes. 
The starting point of this paper is a conjecture by Klee and Novik concerning the face numbers of balanced simplicial polytopes \cite{KN}. 

We first explain this conjecture of Klee and Novik. 
For a simplicial $d$-polytope $P$,
let $f_i(P)$ denote the number of its $i$-dimensional faces for $-1 \leq i \leq d-1$, where $f_{-1}(P)=1$ 
and define the {\em $h$-number} $h_i(P)$ of $P$ by 
$
h_i(P)=\sum_{j=0}^i (-1)^{j-i} {d-j \choose i-j} f_{j-1}(P)
$
for $0 \leq i\leq d$. 
The following Generalized Lower Bound Theorem, which holds for all simplicial polytopes, was originally conjectured by McMullen and Walkup \cite{MW} and later proved in \cite{MN,MW,St1}.

\begin{theorem}[Generalized Lower Bound Theorem]
\label{GLBT}
Let $P$ be a simplicial $d$-polytope. Then
$$h_0(P) \leq h_1(P) \leq \cdots \leq h_{\lfloor \frac d 2 \rfloor}(P).$$
Moreover, 
$h_{i-1}(P)=h_i(P)$ for some $i \leq \frac d 2$
if and only if $P$ is $(i-1)$-stacked, that is,
$P$ can be triangulated without introducing faces of dimension $\leq d-i$.
\end{theorem}

We say that a simplicial $d$-polytope is {\em balanced} if its underlying graph is $d$-colorable.
Inspired by Theorem \ref{GLBT},
Klee and Novik \cite[Conjecture 5.5]{KN} proposed the following balanced analogue. 

\begin{conjecture}[Balanced Generalized Lower Bound Conjecture]
\label{bglbc}
Let $P$ be a balanced simplicial $d$-polytope. Then
$$\frac {h_0(P)} {{d\choose 0}} \leq \frac {h_1(P)} {{d\choose 1}} \leq
\cdots \leq \frac {h_{\lfloor \frac d 2 \rfloor}(P)} {{d\choose \lfloor \frac d 2 \rfloor}}.$$
Moreover,
one has $\frac {h_{i-1}(P)} {{ d \choose i-1}}= \frac {h_{i}(P)} { {d\choose i}}$
for some $i \leq \frac d 2$ if and only if $P$ has the balanced $(i-1)$-stacked property.
\end{conjecture}

We refer the reader to \cite[Definition 5.3]{KN} for the definition of the balanced stacked property. 

It is easy to see that the first inequality, namely, $h_0(P) \leq \frac {h_1(P)} d$ indeed holds.
Goff, Klee and Novik \cite{GKN} proved $(d-1) h_1(P) \leq 2 h_2(P)$, which implies the  second inequality of the conjecture
(see also \cite{BK} and \cite{KN} for further generalizations).  

Moreover, the ``if'' part of the equality case was proved in \cite[Theorem 5.8]{KN}, and the ``only if'' part could be verified for $i\leq 2$ \cite[Theorem 4.1]{KN}. 
In this paper, we give an affirmative answer to the first part of the conjecture.

For a balanced simplicial simplicial $d$-polytope $P$, let 
$$\overline g_i(P) = i h_i(P)-(d-i+1)h_{i-1}(P),$$
be the \emph{balanced $g$-numbers}, as introduced in \cite{KN}. 
Then it is easy to see that 
$\frac {h_{i-1}(P)} {{ d \choose i-1}}\leq \frac {h_{i}(P)} { {d\choose i}}$
if and only if $\overline g_i(P) \geq 0$.
Thus the next result proves the first part of Conjecture \ref{bglbc}.

\begin{theorem}
\label{thm1.2}
For any balanced simplicial $d$-polytope $P$, one has
$\overline g_i(P) \geq 0$ for all $i=0,1,\dots, \lfloor \frac d 2 \rfloor$.
\end{theorem}

In order to prove Theorem \ref{thm1.2}, we 
study so-called rank-selected subcomplexes.
For a balanced simplicial $d$-polytope $P$, whose vertices are colored by elements in $[d]=\{1,2,\dots,d\}$, and for $T \subseteq [d]$
let $P_T$ be the rank-selected subcomplex of $P$ 
(see Section \ref{sect:Basics} for a precise definition).
We will deduce Theorem \ref{thm1.2} from the following statement for rank-selected subcomplexes.

\begin{theorem}
\label{thm1.3}
Let $P$ be a balanced simplicial $d$-polytope. 
Then for any $T \subseteq [d]$
\begin{itemize}
\item[(i)] $h_i(P_T) \leq h_{\#T -i} (P_T)$ for all $i \leq \frac {\#T} 2$.
\item[(ii)] $h_0(P_T) \leq h_1(P_T) \leq \cdots \leq h_{\lfloor \frac {\#T+1} 2 \rfloor} (P_T)$.
\end{itemize}
\end{theorem}

To prove the above theorem we actually show that Stanley-Reisner rings of rank-selected subcomplexes exhibit a Lefschetz property (Theorem \ref{3.2}). 
This gives a partial affirmative answer to a question posed by Bj\"orner and Swartz in \cite[Problem 4.2]{Sw}.

It follows from the $g$-theorem \cite[III Theorem 1.1]{St} that if $P$ is a simplicial $d$-polytope,
then $h_{i-2}(P)=h_{i-1}(P)$ forces $h_{i-1}(P)=h_i(P)$ for $i \leq \frac d 2$. 
We prove that a similar property holds for balanced $g$-numbers.

\begin{theorem}
\label{thm1.4}
Let $P$ be a balanced simplicial $d$-polytope.
If $\overline g_{i-1}(P)=0$ for some $i \leq \frac d 2$,
then $\overline g_{i}(P)=0$.
\end{theorem}

Note that the above theorem also gives further evidence for the equality case of Conjecture \ref{bglbc}. Indeed, if the conjecture is true, then it would also imply the statement of the previous theorem. 
 
The paper is structured as follows. In Section \ref{sect:Basics} we provide the necessary background on balanced simplicial complexes and study algebraic properties of their Stanley-Reisner rings. Those results will then be employed in Section \ref{sect:Proofs} to provide the proofs of the main results Theorem \ref{thm1.2}, Theorem \ref{thm1.3} and Theorem \ref{thm1.4}.

\section{Stanley-Reisner rings of balanced simplicial complexes}\label{sect:Basics}

In this section, we recall some basic properties of balanced simplicial complexes,
and study algebraic properties of their Stanley-Reisner rings. 

We first recall basic definitions on simplicial complexes.
Let $\Delta$ be a (finite abstract) simplicial complex on the vertex set $[n]=\{1,2,\dots,n\}$.
Thus $\Delta$ is a collection of subsets of $[n]$ satisfying that
$F \in \Delta$ and $G \subset F$ imply $G \in \Delta$.
Elements of $\Delta$ are called {\em faces} of $\Delta$
and maximal faces (under inclusion) are called {\em facets}.
The {\em dimension} of a face $F \in \Delta$ is $\dim F=\#F-1$, where $\#X$ denotes the cardinality of a finite set $X$,
and the {\em dimension} of $\Delta$ is the maximum dimension of its faces.
Faces of dimension $0$ are called {\em vertices} and faces of dimension $1$ are called {\em edges}.
We define the $f$-numbers $f_{-1}(\Delta),\dots,f_{d-1}(\Delta)$ 
and $h$-numbers $h_0(\Delta),\dots,h_d(\Delta)$ of a $(d-1)$-dimensional simplicial complex $\Delta$
in the same way as for simplicial polytopes.

We say that a $(d-1)$-dimensional simplicial complex $\Delta$ on $[n]$ is {\em balanced} (completely balanced in some literature)
if 
its graph is $d$-colorable,
that is,
there is a map $\kappa :[n] \to [d]$, called a coloring map of $\Delta$,
such that $\kappa(x) \ne \kappa(y)$ for any edge $\{x,y\} \in \Delta$.
For the rest of this paper, 
we will assume that $\Delta$ is endowed with a fixed coloring map $\kappa:[n] \to [d]$ 
and we will always use this coloring map. 
For a subset $S \subseteq [d]$, we define
$$f_S(\Delta)=\# \{F \in \Delta: \kappa(F) = S\},$$
where $f_\emptyset(\Delta)=1$,
and
$$h_S(\Delta)= \sum_{T \subseteq S} (-1)^{\#S - \#T} f_S(\Delta).$$
The vectors $(f_S(\Delta) : S \subseteq [d])$
and $(h_S(\Delta): S \subseteq [d])$
are called the {\em flag $f$-vector} and the {\em flag $h$-vector} of $\Delta$, respectively.
For $T \subseteq [d]$, the simplicial complex
$$\Delta_T=\{F \in \Delta: \kappa(F) \subseteq T\}$$
is called the {\em rank-selected subcomplex} of $\Delta$.
It is easy to see that $f_S(\Delta)=f_S(\Delta_T)$ and $h_S(\Delta)=h_S(\Delta_T)$ if $S \subseteq T$.
Also, the usual $f$-numbers and $h$-numbers can be recovered from their flag counterparts by 
$f_{i-1}(\Delta)=\sum_{\# S=i} f_S(\Delta)$ and $h_i(\Delta)=\sum_{\#S =i} h_S(\Delta)$.

For a face $F \in \Delta$,
the subcomplexes
$$\lk_\Delta(F)=\{ G \in \Delta: F \cup G \in \Delta,\ F \cap G = \emptyset\}$$
and
$$\st_\Delta(F)=\{ G \in \Delta: F \cup G \in \Delta\}$$
are called the {\em link} of $F$ in $\Delta$ and 
the {\em star} of $F$ in $\Delta$, respectively.
We say that $\Delta$ is {\em Gorenstein*} (over a field $\KK$) if, for any face $F \in \Delta$ (including the empty face $\emptyset$),
$\lk_\Delta(F)$ has the same $\KK$-homology as a $(d-1-\#F)$-sphere.
The following symmetry of flag $h$-vectors of Gorenstein* balanced simplicial complexes is well-known (see \cite[Corollary 4.7]{BB}).

\begin{lemma}
\label{2.1}
If $\Delta$ is a Gorenstein* balanced simplicial complex of dimension $d-1$,
then $h_S(\Delta)=h_{[d]\setminus S}(\Delta)$ for all $S \subseteq [d]$.
\end{lemma}

Next, we recall Stanley-Reisner rings.
Let $\Delta$ be a $(d-1)$-dimensional balanced simplicial complex on $[n]$
and let $R=\KK[x_1,\dots,x_n]$ be a polynomial ring over an infinite field $\KK$.
The ring
$$\KK[\Delta]=R/(x_F: F \subseteq [n],\ F \not \in \Delta),$$
where $x_F=\prod_{i \in F} x_i$,
is called the {\em Stanley-Reisner ring of $\Delta$}.
The rings $R$ and $\KK[\Delta]$ have a nice $\ZZ^d$-graded structure induced by the coloring map $\kappa$. 
More precisely, for $1\leq i\leq n$, we define $\deg x_i=\ee_{\kappa(i)}$,
where $\ee_1,\dots,\ee_d$ are the unit vectors of $\ZZ^d$. 
For a $\ZZ^d$-graded $R$-module $M$, 
we write $M_\aaa$ for the graded component of $M$ of degree $\aaa$.

Let $I \subset R$ be a homogeneous ideal and let $A=R/I$.
The {\em Krull dimension} of $A$ is the minimal number $k$
such that there is a sequence $\theta_1,\dots,\theta_k \in R$ of linear forms 
such that $\dim_\KK A/(\theta_1,\dots,\theta_k)A< \infty$.
It is well-known that the Krull dimension of $\KK[\Delta]$ equals $\dim \Delta +1$ \cite[II Theorem 1.3]{St}.
If $A$ is of Krull dimension $d$, then a sequence $\Theta=\theta_1,\dots,\theta_d$ of linear forms such that
$\dim_\KK A/\Theta A< \infty$ is called a {\em linear system of parameters} (l.s.o.p.\ for short) of $A$.
We say that a simplicial complex $\Delta$ is {\em Cohen-Macaulay} (over $\KK$)
if the ring $\KK[\Delta]$ is Cohen-Macaulay, that is, any l.s.o.p.\ of $\KK[\Delta]$
is a regular sequence of $\KK[\Delta]$. 
By Reisner's criterion \cite[II Corollary 4.2]{St},
Gorenstein* simplicial complexes are Cohen-Macaulay.
The following results due to Stanley (see \cite{St0} or \cite[III Section 3]{St}) will be of importance later on.

\begin{lemma}
\label{2.2}
Let $\Delta$ be a $(d-1)$-dimensional balanced simplicial complex on $[n]$
and let $\displaystyle{\theta_i=\sum_{v \in [n],\ \kappa(v)=i} x_v}$ for $i=1,2,\dots,d$.
Then,
\begin{itemize}
\item[(i)] we have an equality of formal power series in variables $t_1,\dots,t_d$
\begin{align*}
\sum_{\aaa \in \ZZ^d}\big( \dim_\KK \big(\KK[\Delta]\big)_\aaa \big) \ttt^\aaa
= \frac 1 {(1-t_1)(1-t_2) \cdots (1-t_d)} \left\{ \sum_{S \subseteq [d]} h_S(\Delta) \big( \textstyle \prod_{i \in S} t_i\big) \right\},
\end{align*}
where $\ttt^\aaa=t_1^{a_1} \cdots t_d^{a_d}$
for $\aaa=(a_1,\dots,a_d) \in \ZZ^d$.
\item[(ii)] for any variable $x_v$ with $\kappa(v)=i$,
$x_v^2$ is equal to zero in $\KK[\Delta]/(\theta_i \KK[\Delta])$.
\item[(iii)]
$\theta_1,\dots,\theta_d$ is an l.s.o.p.\ of $\KK[\Delta]$.
\item[(iv)]
if $\Delta$ is Cohen-Macaulay,
so are its rank-selected subcomplexes $\Delta_T$ for $T\subseteq [d]$.
\end{itemize}
\end{lemma}

In the remaining part of this section,
we will prove algebraic properties of Stanley-Reisner rings of balanced simplicial complexes,
which will play a crucial role in the proofs of the main theorems. 
Let $\Delta$ be a $(d-1)$-dimensional balanced simplicial complex on $[n]$ and let $A=\KK[\Delta]$.
Fix an integer $p$ with $0 \leq p \leq d$, and let
$$S=\{1,2,\dots,p\} \mbox{ and } T=\{p+1,\dots,d\}.$$
Let $\Theta=\theta_1,\dots,\theta_p,\theta_{p+1}',\dots,\theta_d'$ be an l.s.o.p.\ of $A$
satisfying the following conditions
\begin{itemize}
\item $\displaystyle{\theta_i=\sum_{v \in [n],\ \kappa(v)=i} x_v}$ for $i=1,2,\dots,p$.
\item $\theta_i' \in \KK[x_v: \kappa(v) \in T]$ for $i=p+1,\dots,d$.
\end{itemize}
(Note that such an l.s.o.p.\ exists by Lemma \ref{2.2}(iii).)
We require these somewhat technnical conditions since, in Section \ref{sect:Proofs}, we will need to choose $\theta_{p+1}',\dots,\theta_d'$
as generic linear forms in $\KK[x_v: \kappa(v) \in T]$.
For simplicity, let $\Theta_S=\theta_1,\dots,\theta_p$
and $\Theta_T'=\theta_{p+1}',\dots,\theta_d'$.

We consider the $\ZZ^{p+1}$-grading of $R$ (and $A$) defined by
$\deg x_v= \ee_{\kappa(v)}$ if $\kappa(v) \in S$ and $\deg x_v=\ee_{0}$ if $\kappa(v) \in T$,
where $\ee_0,\dots,\ee_p$ are the unit vectors of $\ZZ^{p+1}$.
Throughout the rest of this section, we only work with this grading rather than the $\ZZ^d$-grading used in Lemma \ref{2.2}.
Note that $\theta_1,\dots,\theta_p,\theta_{p+1}',\dots,\theta_d'$ are homogeneous with respect to this grading.
Let $M$ be a finitely generated $R$-module having the just defined $\ZZ^{p+1}$-graded structure.
As we did for the usual $\ZZ^d$-grading, we use $M_\aaa$ to denote the graded component of $M$ of degree $\aaa$, where $\aaa\in\ZZ^{p+1}$.
Similarly, we write $M_{\geq \aaa} = \bigoplus_{\bb \geq \aaa} M_\bb$,
where $\bb=(b_0,\dots,b_p) \geq (a_0,\dots,a_p)=\aaa$ if $b_i \geq a_i$ for all $i$. 
Moreover,  $M(\aaa)$ refers to the graded module $M$ with grading shifted by degree $\aaa$ so that $M(\aaa)_\bb=M_{\bb+\aaa}$. Note that $M_{\geq \aaa}$ is a submodule of $M$. 
The ($\ZZ^{p+1}$-graded) {\em Hilbert series of $M$} is the formal power series in variables $t_0,t_1,\dots,t_p$ defined by
$$ \Hilb(M;t_0,t_1,\dots,t_p)
= \sum_{\aaa \in \ZZ^{p+1}} ( \dim_\KK M_\aaa) \ttt^\aaa,$$
where $\ttt^{\aaa}=t_0^{a_0} \cdots t_p^{a_p}$ for $\aaa=(a_0,\dots,a_p)\in\ZZ^{p+1}$. 
For $X \subseteq S$, we set $\ttt^X=\prod_{i \in X} t_i$ and $\ee_X= \sum_{i \in X} \ee_i$.
By Lemma \ref{2.2} (i),
the Hilbert series of $A=\KK[\Delta]$ is given by
\begin{align}
\label{2-1}
\Hilb(A;t_0,\dots,t_p)= \frac 1 { (1-t_0)^{d-p} \prod_{i=1}^p (1-t_i)}
\left\{ \sum_{X \subseteq S,\ Y \subseteq T} h_{X \cup Y} (\Delta) \ttt^X t_0^{\#Y} \right\}.
\end{align}
For a $\ZZ^{p+1}$-graded $R$-module $M$,
we write $M^\vee$ for the (graded) Matlis dual of $M$ \cite[I Section 12]{St}.
Note that if $\dim_\KK M < \infty$, then $M^\vee$ is isomorphic to $\Hom_\KK(M,\KK)$
(we only consider this case in this paper).
From now on we assume that 
all maps are degree preserving $R$-homomorphisms.

\begin{lemma}
\label{2.3}
With the same notation as above,
the following properties hold.
\begin{itemize}
\item[(i)] Let
$$N=\bigoplus_{F \in \Delta,\ \kappa(F)=S} \big( \KK[\lk_\Delta(F)]/\big(\Theta'_T \KK[\lk_\Delta(F)]\big)\big).$$
There exists a surjection of modules 
$ \psi: N(-\ee_S) \to (A/\Theta A)_{\geq \ee_S}.$
\item[(ii)]
If $\Delta$ is Cohen-Macaulay, then
$$\Hilb(A/\Theta A;t_0,\dots,t_p)=
\sum_{X \subseteq S,\ Y \subseteq T} h_{X \cup Y}(\Delta) \ttt^X t_0^{\#Y}.$$
\item[(iii)]
Let $A_T=\KK[\Delta_T]$.
If $\Delta$ is Gorenstein*, then
$$(A_T/(\Theta_T' A_T))^\vee \cong ((A/\Theta A)_{\geq \ee_S}) (\ee_S+(d-p)\ee_0).$$
\end{itemize}
\end{lemma}

\begin{proof}
(i)
For any $F \in \Delta$ with $\kappa(F)=S$, let
$$\varphi_F: \KK[\st_\Delta(F)] (-\ee_S) \to A_{\geq \ee_S}$$
be the map defined by $\varphi_F(\alpha)=x_F \alpha$.
Note that $\varphi_F$ is well-defined since $x_F x_G =0$ in $A=\KK[\Delta]$ for any $G \not \in \st_\Delta(F)$.
Since $A_{\geq \ee_S}$ is generated by monomials $x_F$ with $F\in \Delta$ and with $\kappa(F)=S$,
the sum of the maps $\varphi_F$
$$\bigoplus_{F \in \Delta,\ \kappa(F)=S} \KK[\st_\Delta(F)] (-\ee_S)\to A_{\geq \ee_S}$$
is surjective.
By composing the above map with the natural surjection $A_{\geq \ee_S} \to (A/\Theta A)_{\geq \ee_S}$, we obtain a surjection
$$\varphi: \bigoplus_{F \in \Delta,\ \kappa(F)=S} \KK[\st_\Delta(F)] (-\ee_S) \to (A/\Theta A)_{\geq \ee_S}.$$
Since, by Lemma \ref{2.1}(ii), $\varphi_F (x_i)=x_ix_F=0$ for any $i \in F$,  
the kernel of $\varphi$ contains 
$$\bigoplus_{F \in \Delta,\ \kappa(F)=S} \big((x_i:i \in F)+(\Theta_T')\big)\KK[\st_\Delta(F)].$$
Using that
$$\KK[\st_\Delta(F)]/\big((x_i:i\in F)\KK[\st_\Delta(F)]\big)\cong \KK[\lk_\Delta(F)],$$
we conclude that the map $\varphi$ induces a surjection from $N$ to $(A/\Theta A)_{\geq \ee_S}$.

(ii)
Since $A$ is Cohen-Macaulay,
$\Theta_S$ is a regular sequence of $A$.
Then the short exact sequence
$$0 \longrightarrow A/((\theta_1,\dots,\theta_{i-1})A) (-\ee_i)
\stackrel{\times \theta_i} \longrightarrow
A/((\theta_1,\dots,\theta_{i-1})A)
\longrightarrow
A/((\theta_1,\dots,\theta_{i})A)
\longrightarrow 0$$
shows that
\begin{align}
\label{2-1-2}
\Hilb(A/\Theta_S A;t_0,\dots,t_p)&=
(1-t_1) \cdots (1-t_p)\Hilb(A;t_0,\dots,t_p).
\end{align}
Similarly, since $\Theta_T'$ is a regular sequence of $A/\Theta_SA$, we have
\begin{align*}
\Hilb(A/\Theta A;t_0,\dots,t_p)&=
(1-t_0)^{d-p}\ \! \Hilb(A/\Theta_SA ;t_0,\dots,t_p)\\
&= \sum_{X \subseteq S,\ Y \subseteq T} h_{X\cup Y}(\Delta) \ttt^X t_0^{\#Y},
\end{align*}
where we use \eqref{2-1} and \eqref{2-1-2} for the second equality.

(iii) Recall that for any $\ZZ^{p+1}$-graded $R$-module $M$, one has
\begin{align}
\label{hoshi}(M^\vee)_\aaa \cong \Hom_\KK(M_{-\aaa},\KK) \cong M_{-\aaa}
\end{align}
as $\KK$-vector spaces.
Since $\Delta$ is Gorenstein*,
$A/\Theta A$ is a Gorenstein graded algebra of Krull dimension $0$.
Thus $A/\Theta A$ and $(A/\Theta A)^\vee$ are isomorphic up to a certain shift in grading \cite[I Theorem 12.5]{St}.
Since $h_{[d]}(\Delta)=h_\emptyset(\Delta)=1$, \eqref{hoshi} and the formula for the Hilbert series of $A/\Theta A$ stated in part (ii) shows that this shift of degree must be $\ee_S +(d-p)\ee_0$, that is,
$$(A/\Theta A)^\vee \cong (A/\Theta A) (\ee_S +(d-p)\ee_0).$$
Then the natural surjection
$$A/\Theta A \to A/\big(\big((x_v: \kappa(v) \in S)+(\Theta)\big)A \big) \cong A_T/(\Theta_T' A_T)$$
induces an injection
\begin{align}
\label{2-2}
(A_T/(\Theta_T' A_T))^\vee \to (A/\Theta A)^\vee \cong (A/\Theta A) (\ee_S +(d-p)\ee_0).
\end{align}
Since $A/\Theta A$ has Krull dimension $0$,
so has $A_T/(\Theta_T' A_T)$. 
Using that $A_T$ has Krull dimension $\dim \Delta_T +1 = \#T$, we can thus conclude that 
$\Theta_T'$ is an l.s.o.p.\ of $A_T$.
Since $\Delta_T$ is Cohen-Macaulay, we infer from 
 part (ii) (applied to the situation that $A=A_T$ and $S= \emptyset$)
$$\Hilb(A_T/(\Theta_T' A_T);t_0,\dots,t_p)= \sum_{Y \subseteq T} h_Y(\Delta_T) t_0^{\#Y} = \sum_{Y \subseteq T} h_Y(\Delta) t_0^{\#Y}.$$
In particular, by \eqref{hoshi} we have
\begin{align}
\label{2-3}
\Hilb\big(\big(A_T/(\Theta_T' A_T)\big)^\vee ;t_0,\dots,t_p\big) = \sum_{Y \subseteq T} h_Y(\Delta) t_0^{-\#Y}.
\end{align}
This implies $(A_T/(\Theta_T' A_T))^\vee=(A_T/(\Theta_T' A_T))^\vee_{\geq -(d-p)\ee_0}$.
Then since
$$( A/\Theta A)(\ee_S+(d-p)\ee_0)_{\geq -(d-p)\ee_0}
=((A/\Theta A)_{\geq \ee_S})(\ee_S +(d-p)\ee_0),$$
the map \eqref{2-2} induces an injection
$$\big(A_T/(\Theta_T'A_T)\big) ^\vee
\to ((A/\Theta A)_{\geq \ee_S})(\ee_S +(d-p)\ee_0).$$
Now, to obtain the desired isomorphism,
it is enough to show that the module
$((A/\Theta A)_{\geq \ee_S})(\ee_S +(d-p)\ee_0)$
has the same Hilbert series as $(A_T/(\Theta_T'A_T))^\vee$, which has been computed in \eqref{2-3}.
Indeed, part (ii) and Lemma \ref{2.1} say that the Hilbert series of $(A/\Theta A)_{\geq \ee_S}$ equals
$$\textstyle
\sum_{Y \subseteq T} h_{S \cup Y}(\Delta) \ttt^S t_0^{\#Y}
=\sum_{Y \subseteq T} h_{T \setminus Y}(\Delta) \ttt^S t_0^{\#Y},$$
and therefore we have
\begin{align*}
\Hilb\big(\big((A/\Theta A)_{\geq \ee_S} \big) (\ee_S +(d-p)\ee_0 );t_0,\dots,t_p\big)
= \sum_{Y \subseteq T} h_{T\setminus Y} (\Delta) t_0^{\#Y -\#T}
\end{align*}
which is equal to the right-hand side of \eqref{2-3}, as desired.
\end{proof}

\section{Proofs of the main results}\label{sect:Proofs}

In this section, we provide the proofs of the results listed in the introduction.
Throughout this section,
we consider the standard $\ZZ$-grading of the polynomial ring $R=\KK[x_1,\dots,x_n]$ defined by $\deg x_i=1$ for all $i$.
For a $\ZZ$-graded $R$-module $M$, let $M_k$ be the graded component of $M$ of degree $k$
and let $M(k)$ be the graded module $M$ with grading shifted by $k$.

We say that a $(d-1)$-dimensional Gorenstein* simplicial complex $\Delta$ has the {\em strong Lefschetz property} (SLP for short) over $\KK$ if there exist  an l.s.o.p.\ $\Theta$ of $\KK[\Delta]$ and a linear form $\omega$ such that the multiplication map
$$\times \omega^{d-2i} : \big(\KK[\Delta]/(\Theta \KK[\Delta])\big)_i \to \big(\KK[\Delta]/(\Theta \KK[\Delta])\big)_{d-i}$$
is an isomorphism for all $i\leq \frac d 2$.
A linear form $\omega$ satisfying the above condition is called a {\em Lefschetz element} of $\KK[\Delta]/(\Theta \KK[\Delta])$.
It is known that if $\Delta$ has the SLP, then, for a generic choice of linear forms $\theta_1,\dots,\theta_d,\omega$, the sequence
$\Theta=\theta_1,\dots,\theta_d$ is an l.s.o.p.\ of $\KK[\Delta]$
and
$\omega$ is a Lefschetz element of $\KK[\Delta]/(\Theta \KK[\Delta])$ (see \cite[Proposition 3.6]{Sw}).
Throughout this section,
we regard a simplicial polytope $P$ as a simplicial complex by identifying $P$
with its boundary complex. 
By this identification, every simplicial $d$-polytope is a Gorenstein* simplicial complex of dimension $d-1$.
The following result follows from the Hard Lefschetz theorem for projective toric varieties.
See \cite[III Section 1]{St}.

\begin{lemma}\label{Stanley:Lefschetz}
A simplicial $d$-polytope has the SLP over $\QQ$.
\end{lemma}

We also recall the following well-known fact (see e.g., \cite[II Section 3]{St}).

\begin{lemma}
\label{cm}
Let $\Delta$ be a simplicial complex and let $\Theta$ be an l.s.o.p.\ of $\KK[\Delta]$.
If $\Delta$ is Cohen-Macaulay over $\KK$, then $\dim_\KK (\KK[\Delta]/(\Theta \KK[\Delta]))_i=h_i(\Delta)$ for all $i$.
\end{lemma}

The following statement is crucial for the proofs of Theorems \ref{thm1.2}, \ref{thm1.3} and \ref{thm1.4}.

\begin{theorem}
\label{3.2}
Let $P$ be a balanced simplcial $d$-polytope, $T \subseteq [d]$ and $A_T=\QQ[P_T]$.
There exist an l.s.o.p.\ $\Theta_T$ of $A_T$ and a linear form $\omega$ such that
the multiplication
$$\times \omega ^{\#T-2i} : \big(A_T/(\Theta_T A_T)\big)_i \to \big(A_T/(\Theta_T A_T)\big)_{\#T-i}$$
is injective for $i \leq \frac {\# T} 2$. 
\end{theorem}

\begin{proof}
Let $S=[d] \setminus T$.
We may assume that $S=\{1,2,\dots,p\}$ and $T=\{p+1,\dots,d\}$ for some $0 \leq p \leq d$.
Let $\Theta_S=\theta_1,\dots,\theta_p$ with $\theta_i=\sum_{\kappa(v)=i} x_v$ for $i=1,2,\dots,p$.
Since a link in a simplicial polytope is again a simplicial polytope,
for a sufficiently generic choice of linear forms $\theta_{p+1}',\dots,\theta_{d}',\omega$ in $\QQ[x_v: \kappa(v) \in T]$,
we have
\begin{itemize}
\item[(a)] $\Theta_T'=\theta_{p+1}',\dots,\theta_d'$ is an l.s.o.p.\ of $\QQ[\lk_P(F)]$ and $\omega$ is a Lefschetz element of $\QQ[\lk_P(F)]/(\Theta_T' \QQ[\lk_P(F)])$ for any face $F$ of $P$ with $\kappa(F)=S$. 
\item[(b)] $\Theta_T'$ is an l.s.o.p.\ of $\QQ[P_T]$.
\item[(c)] $(\Theta_S,\Theta_T')$ is an l.s.o.p.\ of $\QQ[P]$.
\end{itemize}
((b) and (c) follow from \cite[III Lemma 2.4]{St}.) 
Let $M=(\QQ[P]/((\Theta_S,\Theta_T') \QQ[P]))_{\geq \ee_S}$
and
$\displaystyle{N=\bigoplus_{F \in P,\ \kappa(F)=S} \QQ[\lk_P(F)]/(\Theta_T' \QQ[\lk_P(F)])}$.
By Lemma \ref{2.3} (i), there exists a surjection $\psi: N(-\#S) \to M$.
Consider the commutative diagram
\begin{eqnarray*}
\begin{array}{cccccc}
N_{\#T-i} & \stackrel \psi \longrightarrow & M_{\# T-i+\#S}\medskip\\
\times \omega^{\#T -2i} \uparrow\hspace{40pt} & & \hspace{40pt} \uparrow \times \omega^{\#T -2i}\smallskip\\
N_{i} & \stackrel \psi \longrightarrow & M_{i+\#S},
\end{array}
\end{eqnarray*}
where $i \leq \frac {\#T} 2$.
Since $\psi$ and the left vertical map are surjective by (a),
the multiplication
\begin{align*}
\times \omega^{\#T -2i} : M_{i+ \#S} \longrightarrow M_{\# T-i+\#S}=M_{d-i}
\end{align*}
is surjective for $i \leq \frac {\#T} 2$.
Since, by Lemma \ref{2.3} (iii), $M$ and $(A_T/(\Theta_T' A_T))^\vee$ are isomorphic up to a shift of degree $d$, the above surjectivity implies that the multiplication
$$\times \omega^{\#T -2i} : \big( A_T/(\Theta_T' A_T)\big)^\vee _{i+\#S-d}
\to \big( A_T/(\Theta_T' A_T)\big)^\vee _{-i}
$$
is surjective for $i \leq \frac {\#T} 2$.
Since $i+\#S-d=-(\#T-i)$,
the desired injectivity follows from the fact that $(L^\vee)^\vee \cong L$ for any graded $R$-module $L$.
\end{proof}

\begin{remark}
It was asked by Bj\"orner and Swartz \cite[Problem 4.2]{Sw} if the conclusion of Theorem \ref{3.2} holds for all doubly Cohen-Macaulay complexes. 
As rank-selected subcomplexes of balanced simplicial polytopes are doubly Cohen-Macaulay, Theorem \ref{3.2} gives a partial affirmative answer to this question. It was shown by Swartz that the conjecture is true for simplicial complexes, which have convex ear decompositions \cite[Theorem 3.9]{Sw}. However, it is not known at present if rank-selected subcomplexes $P_T$ have convex ear decompositions.
\end{remark}

\begin{remark}
Although Theorem \ref{3.2} does not show that rank-selected subcomplexes possess the strong or weak Lefschetz property in the sense of \cite{Book}, it does    show that those complexes have a similar (though weaker) property. 
This type of property is often referred to as Lefschetz properties.
\end{remark}


We can now provide the proof of Theorem \ref{thm1.3}.

\begin{proof}[Proof of Theorem \ref{thm1.3}]
We use the same notation as in Theorem \ref{3.2}.
Since $P_T$ is Cohen-Macaulay,
$\dim_\QQ(A_T/(\Theta_T A_T))_k=h_k(P_T)$ for all $k$.
Then the first statement directly follows from the injectivity in Theorem \ref{3.2}.
From this injectivity we also deduce that the multiplication map
$\times w :(A_T/(\Theta_T A_T))_{i-1} \to (A_T/(\Theta_T A_T))_{i}$ 
is injective for $i \leq \frac {\#T+1} 2$, which implies the second statement.
\end{proof}

Next, we prove Theorem \ref{thm1.2}.

It follows from the next lemma that Theorem \ref{thm1.3} implies Theorem \ref{thm1.2}. This was also noted in \cite{GKN} (see discussion after Theorem 5.3 in \cite{GKN}), but we include the proof for completeness.

\begin{lemma}
\label{3.3}
Let $\Delta$ be a $(d-1)$-dimensional balanced simplicial complex.
For $i \leq k \leq d$, we have
$${d-i \choose k-i} h_i(\Delta)= \sum_{T \subseteq [d],\ \#T=k} h_i(\Delta_T).$$
\end{lemma}

\begin{proof}
The assertion follows  from equations $h_i(\Delta)=\sum_{\# S=i} h_S(\Delta)$
and $h_i(\Delta_T)=\sum_{S \subseteq T,\ \#S=i} h_S(\Delta)$
by a routine double counting argument.
\end{proof}

\begin{proof}[Proof of Theorem \ref{thm1.2}]
By Lemma \ref{3.3},
$$ {d-i \choose i-1} h_i(P)= \sum_{T \subseteq [d],\ \#T=2i-1} h_i(P_T)$$
and
$$ {d-i+1 \choose i} h_{i-1}(P)= \sum_{T \subseteq [d],\ \#T=2i-1} h_{i-1}(P_T).$$
Since $i {d-i+1 \choose i} = (d-i+1){d-i \choose i-1}$,
the above equations say
\begin{align*}
{d-i+1 \choose i}\overline{g}_i(P)
&=(d-i+1) \left\{ {d-i \choose i-1}h_i(P)- { d-i+1 \choose i} h_{i-1}(P) \right\}\\
&=(d-i+1) \left( \sum_{T \subseteq [d],\ \#T =2i-1} h_i(P_T)-h_{i-1}(P_T)\right),
\end{align*}
which is non-negative by Theorem \ref{thm1.3}, as desired.
\end{proof}

The above proof clearly implies the following corollary.

\begin{corollary}
\label{3.4}
Let $P$ be a balanced simplicial $d$-polytope and $1\leq i\leq  \lfloor \frac {d+1} 2 \rfloor$. Then, $\overline g_i(P)=0$ if and only if
$h_i(P_T)=h_{i-1}(P_T)$ for all $T \subseteq [d]$ with $\#T=2i-1$.
\end{corollary}

We finally prove Theorem \ref{thm1.4}.
Before giving a proof,
we show the following simple lemma. 
A simplicial complex is said to be \textit{pure}
if all its facets have the same cardinality.

\begin{lemma}\label{lem:balancedG}
Let $\Delta$ be a $(d-1)$-dimensional pure balanced simplicial complex on $[n]$.
For $0\leq i \leq d$, we have
$$\sum_{v\in [n]}\overline{g}_i \big(\lk_\Delta(v) \big)=i\overline{g}_{i+1}(\Delta)+(d-i)\overline{g}_i(\Delta).$$
\end{lemma}

\begin{proof}
By the definition of the balanced $g$-numbers, we have
\begin{align}
\label{3-3}
\sum_{v\in [n]}\overline{g}_i(\lk_\Delta(v))=
i\left(\sum_{v\in [n]}  h_i\big(\lk_\Delta(v)\big) \right)-(d-i) \left(\sum_{v\in[n]}h_{i-1}\big(\lk_\Delta(v)\big) \right).
\end{align}
On the other hand, since $\Delta$ is pure, 
it follows from \cite[Proposition 2.3]{Sw2} that
$$\textstyle \sum_{v\in [n]}h_i\big(\lk_\Delta(v)\big) = (i+1)h_{i+1}(\Delta)+(d-i)h_{i}(\Delta)$$
for all $i$.
By substituting the above equation in \eqref{3-3},
we get the desired equation.
\end{proof}

Now Theorem \ref{thm1.4} directly follows from Lemma \ref{lem:balancedG} and the next proposition.

\begin{proposition}
Let $P$ be a balanced simplicial $d$-polytope and let $v$ be a vertex of $P$.
If $\overline g_i(P)=0$ for some $i \leq \frac d 2$, then $\overline g_i(\lk_P(v))=0$.
\end{proposition}

\begin{proof}
By Corollary \ref{3.4}, what we must prove is 
\begin{align}
\label{3-2}
h_i\big((\lk_P(v))_T\big) - h_{i-1}\big((\lk_P(v))_T\big)=0
\end{align}
for all $T \subseteq [d] \setminus \{ \kappa (v)\}$ with $\#T=2i-1$.
Fix a subset $T \subseteq [d] \setminus \{ \kappa (v)\}$ with $\#T=2i-1$.
We prove \eqref{3-2}.

Consider $\QQ[P_T]$ and $\QQ[(\lk_P(v))_T]$.
Note that both $\QQ[P_T]$ and $\QQ[(\lk_P(v))_T]$ are of Krull dimension $\#T$.
By Theorem \ref{3.2}, there is a linear system of parameters $\Theta$ for both $\QQ[P_T]$ and $\QQ[(\lk_P(v))_T]$, and a linear form $\omega $ such that
\begin{align*}
& \dim_{\QQ}\big(\QQ[P_T]/((\Theta,\omega)\QQ[P_T])\big)_i\\
&=
\dim_{\QQ}\big(\QQ[P_T]/(\Theta\QQ[P_T])\big)_i \! -\!  \dim_{\QQ}\big(\QQ[P_T]/(\Theta\QQ[P_T])\big)_{i-1}\\
&=h_{i}(P_T)-h_{i-1}(P_T),
\end{align*}
where we use the injectivity in Theorem \ref{3.2} for the first equality
and use Lemmas \ref{2.2} (iv) and \ref{cm} for the second equality,
and similarly
$$\dim_{\QQ}\big(\QQ[(\lk_P(v))_T]/\big((\Theta,\omega)\QQ[(\lk_P(v))_T]\big)\big)_i=h_i\big(\big(\lk_P(v)\big)_T\big)-h_{i-1}\big(\big(\lk_P(v)\big)_T\big).$$
Since $(\lk_P(v))_T$ is a subcomplex of $P_T$,
$\QQ[(\lk_P(v))_T]$ is a quotient ring of $\QQ[P_T]$ and
thus
$$
\dim_{\QQ}\big(\QQ[(\lk_P(v))_T]/((\Theta,\omega)\QQ[(\lk_P(v))_T])\big)_i
\leq
\dim_{\QQ}\big(\QQ[P_T]/((\Theta,\omega)\QQ[P_T])\big)_i.$$
As the assumption $\overline g_i(P)=0$ implies  $h_i(P_T)-h_{i-1}(P_T)=0$ by Corollary \ref{3.4}, 
we get the desired equation \eqref{3-2}.
\end{proof}

\noindent
\textbf{Acknowledgments}:
The second author was partially supported by JSPS KAKENHI 25400043.
We would like to thank Steven Klee and Isabella Novik for their helpful comments on the paper.

\end{document}